\documentclass[12pt]{amsart}

\usepackage{amscd}
\usepackage{fullpage}
\usepackage{marvosym}
\usepackage{amsmath}
\usepackage{amsgen}
\usepackage{amsbsy}
\usepackage{amsopn}
\usepackage{amsthm}
\usepackage{amscd}
\usepackage{amsfonts}          % all the nice fonts -- part of amsart
\usepackage{amssymb}            % extra common ams symbols -- not autoloaded
\usepackage{mathrsfs}             % Adds Script Fonts for Curly scripts see \scr below
\usepackage{hyperref}           % click to jump to refernce in .dvi
\usepackage{ifsym}              % lots of symbols, weather, clocks, legends, random
\usepackage{enumerate,enumitem}          % improved enumeration environment control
\usepackage{marvosym}
 \usepackage{color}
\newtheorem{Theorem}{Theorem}
\newtheorem{Lemma}[Theorem]{Lemma}
\newtheorem{Corollary}[Theorem]{Corollary}

\newtheorem{Definition}{Definition}

\newtheorem{Remark}{Remark}

\usepackage{graphicx, amssymb, amsfonts}

\title{Topological rigidity of higher graph manifolds}

\author{No\'e B\'arcenas, Daniel  Juan-Pineda and Pablo Su\'arez-Serrato}

    \date{\today}

 \address{Centro de Ciencias Matem\'aticas. Universidad Nacional Aut\'onoma  de M\'exico  \\Ap.Postal 61-3 Xangari. 58089 Morelia, Michoac\'an  M\'exico }
 \email{barcenas@matmor.unam.mx}
 \urladdr{http://www.matmor.unam.mx /~ barcenas}
\email{daniel@matmor.unam.mx}

  \address{Instituto de Matem\'aticas,
              Universidad Nacional Aut\'onoma de M\'exico,
              Ciudad Universitaria, Coyoac\'an,
              04510. Mexico City, D. F.
              M\'exico}
\urladdr{http://www.matem.unam.mx/PabloSuarezSerrato}
    \email{pablo@im.unam.mx}

\begin{document}

\begin{abstract}
In this short note we prove the Borel conjecture for a family of aspherical manifolds that includes higher graph manifolds. \end{abstract}

\maketitle

\section{Introduction}

The Borel conjecture is a statement about topological rigidity. It states that a homotopy equivalence between two aspherical manifolds is homotopic to a homeomorphism.

A  lot  of  work  in  geometric  topology has  been  done    in  the  last  years with the aim  to  prove  the  Borel  conjecture   using  methods  involving  controlled  topology and  algebraic  $K$-theory. In  particular,  the Borel conjecture was shown by R. Frigerio, J.-F. Lafont, and A. Sisto to hold for the class of graph manifolds studied  in \cite{FLS}.

On  the  other  hand, relationships between several  generalizations  of  the concept of finite  asymptotic  dimension   in  connection with    isomorphism   conjectures, in  algebraic  $K$  and  $L$- theory ,  as  well  as  coarse  versions  of  these   have   been  carried out by G. Carlsson and B. Goldfarb  in \cite{carlsongoldfarbktheorygeom}, \cite{goldfarbweakcoherence}, \cite{carlsongoldfarbequivariantstable}.

The  method  of  proof  of  the Borel  conjecture in this note  uses  these  previous  developments. 

Consider the following construction of smooth $n$--manifolds $M$, for $n\geq3$:

\begin{Definition}\label{mfds}

\begin{enumerate}
\item For every $i = 1, . . . , r$ take a complete finite-volume non-compact pinched negatively curved
    $n_i$--manifold $V_i$, where $2 \leq n_i
    \leq n$.
\item Denote by $M_i$ the compact smooth manifold with boundary obtained by ``truncating the
    cusps'' of $V_i$, i.e. by removing from $V_i$ a {(nonmaximal)} horospherical
    open neighborhood of each cusp.
\item Take fiber bundles $Z_i\to M_i$ with fiber a compact quotient  $N_i$ of an aspherical simply connected Lie group $\widetilde{N_i} $ by the action of a uniform lattice $\Gamma_i$,
    of dimension $n-n_i$, i.e. $N_i$ is diffeomorphic to $\widetilde{N_i}/\Gamma_i$, where $\widetilde{N_i}$ is a simply connected Lie group and
    $\Gamma_i$ is a uniform lattice. 
    
    \item Fix a complete pairing of diffeomorphic boundary components between distinct
    $Z_i$'s, provided one exists, and glue the paired boundary components using diffeomorphisms, to obtain a connected manifold of dimension $n$.
\end{enumerate}

We will call the $Z_i$'s the {\em pieces} of $M$ and whenever $\dim(M_i)=n$, then we say
$Z_i=M_i$ is a {\em pure piece}, (short for purely negatively curved).
\end{Definition}

\begin{Remark}
The construction in the previous definition includes:
\begin{enumerate}[leftmargin=0.25in]
\item The class of
{\bf generalized graph manifolds} of Frigerio,  Lafont and Sisto
\cite{FLS}. The pieces $V_i$ in item $(1)$ above are required to be hyperbolic
with toral  boundary cusps, the $N_i$ in item $(3)$ are required to be tori, and the gluing
diffeomorphisms in item $(4)$ are required to be affine diffeomorphims. These authors produce
examples of manifolds within this class that do not admit any $CAT(0)$ metric. 

\item The family of {\bf cusp-decomposable manifolds} of T. Tam Nguyen Phan
\cite{T}, where interesting (non)rigidity properties are explored. These manifolds only have pure pieces.

\item The {\bf affine  twisted doubles} of hyperbolic manifolds, for which C.S. Aravinda and T. Farrell study in \cite{AF} the existence of nonpositively curved metrics.

\item The {\bf higher graph manifolds} studied in \cite{CS} by C. Connell and the third named author. In that family, item (3) consists of infanilmanifold bundles with affine structure group, which are moreover trivial near the 'cusp boudary' of the negatively curved pieces in the base. In item (4) the glueing diffeomorphisms are restricted to those which are isotopic to affine diffeomorphisms. These two restrictions are used in \cite{CS} to prove statements about collapsing and computations of minimal volume. They turn out not to be needed in the arguments we present for the Borel conjecture to hold true.

\end{enumerate}
\end{Remark}

The following theorem is our main result: 
\begin{Theorem} \label{theoremborel}
Let $M$  be  an  $n$--dimensional manifold constructed as in Definition \ref{mfds}, for $n\geq 6$, then $M$ satisfies the  Borel  conjecture. That  is, given a  homotopy  equivalence $f:M \to M^{'}$,  where  $M^{'}$  is an aspherical  $n$--dimensional manifold,  then $f$ is  homotopic   to  a  homeomorphism. 
\end{Theorem}

The following section explains the notions of asymptotic dimension, weak regular coherence, and finite decomposition complexity. In the last section a proof of Theorem \ref{theoremborel} that uses these properties can be found, and also a proof that presents a slight extension of the general strategy proposed by Frigerio-Lafont-Sisto, and we verify it for the higher graph manifolds whose pieces are trivial bundles.

\subsection*{Acknowledgements}{\small
The  authors  thank Boris  Goldfarb  and  Jim Davis  for  conversations  related  to  this  paper. The  first  named  author  aknowledges  the  support  of  UNAM PAPIIT Grant  ID100315. The second named author has support from UNAM-PAPIIT and CONACyT research grants. The third named author thanks CONACyT Mexico and DGAPA-UNAM for supporting various research initiatives. 
}

\section{Finite  Asymptotic  Dimension and  Weak  Regular  Coherence  }

\subsection{Asphericity}
 Consider the next definition, following \cite{FLS}, which will be used later on:

\begin{Definition}
The boundaries of the pieces $Z_{i}$ that are identified together in Definition \ref{mfds} will be called the {\bf internal walls} of $M$.
\end{Definition}

Now we will prove, via an adaptation of the arguments of Frigerio-Lafont-Sisto, that the manifolds we are interested in are in fact aspherical.

\begin{Lemma}\label{asphericity} 
If $M$ is a manifold (possibly with boundary) constructed as in Definition \ref{mfds}, then $M$ is aspherical.
\end{Lemma}

\begin{proof}
This proof is by induction on the number of internal walls $c$ of $M$. If $c=0$ then $M=Z$ for some bundle $Z$ over a closed, negatively curved base. It follows from the homotopy exact sequence for the bundle $Z$ that $M$ is aspherical in this case, establishing the base case for our inductive argument.

Assume $c>0$, and that the result holds for manifolds constructed as in Defintion \ref{mfds}, with strictly less than $c$ internal walls. Cut open $M$ along an arbitrary internal wall $W$. Our inductive hypothesis implies that now $M$ is obtained by gluing one or two (depending on whether $W$ separates $M$ or not) aspherical spaces. Since the inclusion of $W$ in the piece(s) in $M$ it belongs to is $\pi_{1}$--injective, it follows from a classical result of Whitehead \cite{W39} that $M$ is aspherical. \end{proof}

\subsection{Finite asymptotic dimension}

Let  $G$  be  a  finitely  presented  group.   Fix  a    finite  generator set  $S$  and  consider  the   word  metric $d_S$   induced  by  the  generating  set.  With  this  metric,  the   group  $G$ is  a  proper   metric  space. 
 
 \begin{Definition}
A  family $\{ U\}$  of   subsets  in a  metric  space  $X$ is  $D$--disjoint   if  $d(U, U^{'})>D$  for  all  subsets in  the  family.  The  asymptotic  dimension ${\rm asdim} \, X$   of  $X$ is  the   smallest  number  $n$  such  that  for  any  $D>0$  there  is  a  uniformly  bounded  cover  of  $X$  by  $n+1$-familes  of  $D$-disjoint  families  of  subsets. 
 \end{Definition}
 
 An example of spaces (and groups) for which their asymptotic dimension can be explicitly computed are precisely quotients of simply connected Lie groups:
 
 \begin{Theorem}\label{asdimLieLattice}(Carlsson-Goldfarb, Cor. 3.6 in \cite{goldfarbweakcohdeccomp})Let $\Gamma $ be a cocompact lattice in a connected Lie group $G$ with maximal compact subgroup $K$. Then ${\rm asdim} \, \Gamma = dim (G/K)$.  
\end{Theorem}

 For spaces that are built up using smaller subsets, there is a theorem that allows us to bound the asymptotic dimension of the total space. Let $X$ be a metric space. The family $\{ X_{\alpha} \} $ of subsets of $X$ is said to satisfy the inequality ${\rm asdim} X_{\alpha} \leq n$ {\bf uniformly} if for every $r<\infty$ a constant $R$ can be found so that for every $\alpha$ there exists $R$--disjoint families $U_{\alpha}^{0}, U_{\alpha}^{1}, U_{\alpha}^{2}, \ldots , U_{\alpha}^{n}$ of $R$--bounded subsets of $X_{\alpha}$ covering $X_{\alpha}$. 
 
 \begin{Theorem}\label{asdimUnion}(Union theorem, Bell-Dranishnikov, Thm. 25 in \cite{BD})  Let $X=\bigcup\limits_{\alpha} X_{\alpha}$ be a metric space where the family $\{ X_{\alpha} \} $ satisfies the inequality ${\rm asdim} \, X_{\alpha} \leq n$ uniformly. Suppose further that for every $r$ there is a $Y_{r}\subset X$ with ${\rm asdim} \, Y_{r} \leq n$ so that $d(X_{\alpha} - Y_{r}, X_{\alpha'} - Y_{r} )\geq r$ whenever $X_{\alpha} \neq X_{\alpha'}$. Then ${\rm asdim} \, X \leq n$.
\end{Theorem}

\begin{Lemma}\label{asdimgraphmanifold}
The  fundamental  group  $\pi_1(M)$ of  a manifold $M$ of dimension $n$ constructed as in Definition \ref{mfds} has  finite asymptotic  dimension. 
\end{Lemma}

\begin{proof}
 The  fundamental  groups  of the pieces  $\pi_1(Z_i)$  fit  in  an  exact  sequence 
 $$ 1\to \pi_1(N_{i})\to \pi_1( Z_i) \to \pi_1(M_i) \to 1.  $$
 The    asymptotic  dimension  of  $\pi_1(M_i)$ equals $n-n_i$ by  the  Cartan-Hadamard  Theorem. On  the  other  hand,  the  asymptotic  dimension  of   the fibres, which are quotients of Lie groups $G$ under the action of a uniform lattice, equals  ${\rm dim} (G/K)<\infty$  by  Theorem \ref{asdimLieLattice}.   
 
Finally, we invoke Theorem  \ref{asdimUnion},  from which we  conclude  that  the  asymptotic  dimension  of  $\pi_1(X)$  is  finite.
\end{proof}

\subsection{Finite decomposition complexity}

We will briefly define the notion of {\em straight finite decomposition complexity}, since we use it as a key property in the proof of the main result presented below. 

Let $\mathcal{X}$ and $\mathcal{Y}$ be two families of metric spaces, and $R>0$. The family $\mathcal{X}$ is called $R$--decomposable over $\mathcal{Y}$ if, for any space $X$ in $\mathcal{X}$ there are collections of subsets $\{U_{1,\alpha} \} $ and $\{U_{2,\beta} \} $ such that 
$$X=\bigcup\limits_{i=1,2, \gamma=\alpha, \beta} U_{i,\gamma}.$$
Each $U_{i,\gamma}$ is a member of the family $\mathcal{Y}$, and each of the collections $\{U_{1,\alpha} \}$ and $\{U_{2,\beta} \}$ is $R$--disjoint. A family of metric spaces is called {\em bounded } if there is a uniform bound on the diameters of the spaces in the family.

\begin{Definition}\label{sFDC}
A metric space $X$ has {\em straight finite decomposition complexity} if, for any sequence $R_1 \leq R_2 \leq \ldots $ of positive numbers, there exists a finite sequence of metric families $V_1, V_2, \dots, V_{n}$ such that $X$ is $R_1$--decomposable over $V_1$, $X$ is $R_2$--decomposable over $V_2$, etc, and the family $V_{n}$ is bounded.
\end{Definition}

The following well known lemma ties this notion with that of asymptotic dimension, for completeness we include a proof:

\begin{Lemma}\label{asdimsFDC}
If a group has finite asymptotic dimension, then it has straight finite decomposition complexity.
\end{Lemma}

\begin{proof}
It was shown by Guentner-Tessera-Yu that a countable metric space of finite asymptotic dimension has finite decomposition complexity in \cite{GTY}. As part of their study of straight finite decomposition complexity, Dranishnikov-Zarichnyi showed in \cite{DZ} that groups with finite decomposition complexity have stright finite decomposition complexity.
\end{proof}

\section{Two proofs}

\subsection{The Carlsson-Goldfarb approach to the Borel conjecture}

Let $\Gamma$  be  the  fundamental  group  of  a  manifold constructed as in Defintion \ref{mfds}. The  strategy  for  proving  Theorem \ref{theoremborel} for  manifolds  with  fundamental  group $\Gamma$  consists  of  showing that $\Gamma$ satisfies the following properties: 

\begin{enumerate}
\item   $\Gamma$  has  finite  asymptotic  dimension. 
\item  $\Gamma$  has  a finite  model  for  the  classifying  space $B\Gamma$. \end{enumerate}

A  group  satisfying  these two  conditions has been proven  to  also satisfy the integral    isomorphism  conjecture in  algebraic  K-theory, according to Theorem 3.11 of Goldfarb in \cite{goldfarbweakcoherence}.

\begin{proof}(of Theorem \ref{theoremborel})

Item $(1)$ was shown in  Lemma \ref{asdimgraphmanifold} above.

 Item  $(2)$  follows  from the  fact that these are fundamental groups of  compact aspherical manifolds (possibly with boundary)
 Therefore the Borel conjecture holds for the manifolds in Definition \ref{mfds}. \end{proof}

This simple strategy provides an alternative to the one layed out by Frigerio-Lafont-Sisto in \cite{FLS}. We also present in the following a modified version of their strategy, and verify that it can be carried out for certain manifolds within those of Definition \ref{mfds}.

In a  series  of   articles, B.  Goldfarb  and  G. Carlsson  have  investigated  several  notions  which  generalize  that  of  regular  coherence  for  the  group  ring of  infinite  groups.  The  main  geometric  interest   on  this  situation  resides  on  the  fact  that  these   conditions   are  strong  enough  to    allow  the  vanishing  of  the  Whitehead  group  and  negative  algebraic  K--theory groups  of   group  rings, and  weak  enough  to  be  handled  with  methods dealing with  coarse versions  of  the  isomorphism   conjecture  in   algebraic  $K$--theory \cite{goldfarbweakcoherence},\cite{goldfarbweakcohdeccomp}. 

We  will  recall  some   definitions  and  fundamental  results related  to   finite  asymptotic  dimension  and the  coarse assembly  map  in  the  boundedly controlled  setting. See, for example,  \cite{carlsongoldfarb} and \cite{goldfarbweakcoherence} for  further  reference.

Let  $\mathcal{P} (G)$ be  the   power  set  viewed  as  a  category  where  morphisms  are  inclusions  of  subsets.  Let  $R$  be  a  noetherian ring  and  consider  a  finitely  generated $R[G]$--module  $\mathcal{M}$.    A $G$--filtration of  $\mathcal{M}$  is  a  functor $f:\mathcal{P}(G)\to R-Sub(\mathcal{M}) $  to  the  category  of   $R$--submodules  of  $\mathcal{M}$ such  that $f(G)=\mathcal{M}$,  and  each bounded  set  in  the   word  metric  $d_S$ , $T\subset G$ is  mapped  to  a  finitely  generated  $R$--submodule.  Such  a  functor  $f  $   is  equivariant if  $f(g S)=gf(S)$

\begin{Definition}
A homomorphism $\phi: F_1\to  F_2$  between finitely  generated $R[G]$--modules with  fixed  filtrations  $f_1$, $f_2$  is   boundedly  controlled  with  respect  to  the  bound  $D>0$ if $\phi(f_1(S))\subset f_2(B_D(S))$  for  each subset  $S\subset G$.  If  $\phi$  also satisfies $\phi F_1\cap f_2(S) \subset \phi F_1(B_D(S))$,  then  $F$ is  called  boundedly  bicontrolled.     
\end{Definition}

\begin{Definition}
Let  $\mathcal{M}$  be a finitely  presented  $R[G]$--module. A  finite  presentation  $F:R[G]^{m}\to R[G]^{n}\to \mathcal{M}$  is  admissible  if  the  homomorphism $F$  is  boundedly  bicontrolled.  
\end{Definition}

\begin{Definition}
A  group ring  $R[G]$  is  weakly  coherent  if  every $R[G]$--module  with an  admissible  presentation  has  a  projective resolution of  finite  type. Similarly, a  group  ring  is  weakly  regular  coherent  if  every $R[G]$--module  with  an admissible  presentation  has  finite  homological  dimension.  
\end{Definition}

\begin{Theorem}\label{TheoremFADweakcoherence}
(Carlsson-Goldfarb, Corollary  3.9  in \cite{goldfarbweakcoherence})
Let  $R$ be  a  noetherian  ring  and  let  $G$  be  a  group  of  finite  asymptotic  dimension. Then, the  group  ring $R[G]$ is  weakly  regular  coherent. 

\end{Theorem}

Weak  regular  coherence  has   been  verified  to  be  enough  to  guarantee   the  vanishing  of  Whitehead  groups  and  negative   algebraic  $K$--theory.

\begin{Theorem} \label{TheoremFADAssembly}
(Goldfarb, Theorem  3.11  in  \cite{goldfarbweakcoherence}).
Let $G$  be  a  group  of  finite  asymptotic  dimension (or  more  generally of  finite decomposition  complexity,  as  explained  in  \cite{goldfarbweakcoherence}). Assume  that  there  is  a  finite   model  for  the  classifying  space  $K( G, 1)$.  Then,   the  assembly  map  in  algebraic  $K$--theory is  an  isomorphism.  In  particular,  the  Whitehead group  of  $G$ vanishes. 
 \end{Theorem}

As a consequence of theorem \ref{TheoremFADAssembly} and lemma \ref{asdimgraphmanifold}, we obtain: 
\begin{Corollary}
The  group  ring  $\mathbb{Z} \pi_1(M)$ of a manifold $M$ constructed as in Definition \ref{mfds}  is  weakly  regular  coherent. 
\end{Corollary}

\subsection{An extension of the Frigerio-Lafont-Sisto approach to the Borel conjecture}

The proof of the Borel conjecture for the class of manifolds studied by Frigerio-Lafont-Sisto in \cite{FLS} in fact developed a general strategy to be carried out for a given family of manifolds. In their Theorem 3.1 they proved that if a manifold is built up from a geometric decomposition, as are the higher graph manifolds in this paper, and  satisfies the following six conditions, then it  also satisfies the Borel conjecture:

\begin{enumerate}\label{frigeriolafontsisto}
\item Each of the inclusions $W_{i,j}\to Z_i$ is $\pi_1$--injective.
\item Each of the pieces ${Z_i}$ and each of the walls $W_{i,j}$ are aspherical.
\item Each of the pieces ${Z_i}$ and each of the walls $W_{i,j}$ satisfy the Borel Conjecture.
\item The rings ${\bf Z}\pi_1(W_{i,j})$ are all regular coherent.
\item $Wh_{k}({\bf Z}\pi_1(W_{i,j})) = 0 = Wh_{k}({\bf Z}\pi_1(Z_i))$ for $k \leq 1$.
\item Each of the inclusions $\pi_1(W_{i,j})\to \pi_1(Z_i)$ is square-root-closed.

\end{enumerate}

We propose a slightly modified version of this strategy, where we replace the last three conditions, (4), (5) and (6), by a couple of new requirements. So that we obtain the following:

\begin{Lemma}\label{modifiedstrategy}
Let $M$ be a compact manifold of dimension $n\geq 6$ with a topological decomposition (as described in \cite{FLS}).  Assume  the
following conditions hold:
\begin{enumerate}\label{frigeriolafontsisto}
\item Each of the inclusions $W_{i,j}\to Z_i$ is $\pi_1$--injective.
\item Each of the pieces ${Z_i}$ and each of the walls $W_{i,j}$ are aspherical.
\item Each of the pieces ${Z_i}$ and each of the walls $W_{i,j}$ satisfy the Borel Conjecture.
\item The group $\Gamma=\pi_1(M)$ has finite decomposition complexity.
\item There exists a finite model for the classifying space $K(\Gamma, 1)$. 
\end{enumerate}

Then the manifold $M$ also satisfies the Borel conjecture. 
\end{Lemma}

\begin{proof}
Conditions (4) and (5) imply that the Whitehead groups $Wh_i({\bf Z}\Gamma)=0$, for $i\leq 1$, as proved in \cite{goldfarbweakcoherence}.

Therefore the rest of the proof presented in Theorem 3.1 \cite{FLS} follows through, and the result holds. \end{proof}

Now we will concentrate on certain higher graph manifolds, explained briefly in the introduction (see \cite{CS}).

\begin{Lemma}\label{FIC-for-pieces-and-walls}
Assume $M$ is a higher graph manifold, all of whose pieces are trival as bundles. Then, each of the pieces ${Z_i}\cong N_{i}\times M_i$, and each of the walls $W_{i,j}$, satisfy the fibred isomorphism conjecture (FIC) of Farrell-Jones.
\end{Lemma}
\begin{proof}

First notice that the validity of FIC for the walls $W_{i,j}$ follows from the work of Bartels-Farrell-L\"uck in \cite{BFL}, since these are quotients of Lie groups  (see also their Remark 2.13). 

As each piece $Z_i$ is a trivial fibre bundle
\begin{equation*}
 Z_i \cong N_{i}\times M_i
\end{equation*}
the fundamental group of $Z_i$ is a product 
\begin{equation*}
\pi_1( Z_i) \cong \pi_1(N_{i}) \times \pi_1(M_i).
\end{equation*}

Recall that $M_i$ is a manifold that admits a pinched negatively curved metric. So it also admits a $CAT(0)$ metric, and therefore FIC holds for $\pi_1(M_i)$. The fibres satisfy FIC following \cite{BFL}. Therefore $\pi_1( Z_i)$ also satisfies FIC, by Theorem 2.9 in \cite{BFL}. \end{proof}

As a consequence we obtain that the Borel conjecture holds for each of the pieces ${Z_i}$, with trivial fibration structure, and each of the walls $W_{i,j}$, and so condition (3) is verified.

\begin{Lemma} \label{weakregcoherent}
Let $M$ be a higher graph manifold, all of whose pieces $Z_{i}$ are trvial as bundles, and let $W_{i,j}$ denote its internal walls. Then, the rings ${\bf Z}\pi_1(W_{i,j})$ are  weakly regular coherent. 
\end{Lemma}
\begin{proof}
From  the proof of Lemma \ref{asdimgraphmanifold}, we  conclude  that  these  groups  have  finite  asymptotic  dimension. Now the result follows from Theorem \ref{TheoremFADweakcoherence}. 
\end{proof}

\begin{Lemma}\label{whitehead}
Let $M$ be a higher graph manifold, all of whose pieces $Z_{i}$ are trvial as bundles, and let $W_{i,j}$ denote its internal walls. Then, $Wh_{k}({\bf Z}\pi_1(W_{i,j})) = 0 = Wh_{k}({\bf Z}\pi_1(Z_i))$ for $k \leq 1$.
\end{Lemma}
\begin{proof}
Since each of the walls and pieces are aspherical, their fundamental groups are torsion-free. By the previous Lemma \ref{FIC-for-pieces-and-walls}, the result holds for each of the pieces and walls. Alternatively,  the  result  follows  from  Theorem \ref{TheoremFADAssembly}. 
\end{proof}

\begin{Lemma}\label{squareroot}
Let $M$ be  a  manifold constructed as in Definition \ref{mfds}. For  every  $1\leq i\leq r$,  the  map  $N_i\to X$ is  $\pi_1$--injective. Moreover,  the  image  of $\pi_1(N_i$)  is  a  square  root  closed subgroup   in  the  group $\pi_1(X)$.  
\end{Lemma}

\begin{proof}
 Consider  the  long  exact  sequence  of  homotopy  groups   of  a  fibration:
 $$ \ldots \to \pi_{n}( N_{i}) \to \pi_{n}(Z_{i}) \to \pi_n(M_i)\to \ldots . $$
The  connectedness  of   $N_i$  implies  the   $\pi_1$--injectivity  condition. 

Using  proposition VII.2  in  page  168  of  \cite{capellsplitting},  it  suffices  to   verify  the  square root   closed  condition   in  the  fundamental  groups  of  the  edges $ \pi_1(W_i,j) \to \pi_1(Z_i)$. Using  the  long  exact  sequence   of  the  fibration  again,  this  is  equivalent  to   showing  that  there  are  no  $2$--torsion  elements  in  $\pi_1(M_i)$. This  is  certainly  the  case,  since   $M_i$  is  an  aspherical  manifold.   
\end{proof}

\begin{Lemma}\label{finitemodel}
Let $M$ be  a  manifold constructed as in Definition \ref{mfds} and $\Gamma=\pi_1(M)$. Then there exists a finite model for $K(\Gamma, 1)$. 
\end{Lemma}

\begin{proof}
Notice that the manifold $M$ is aspherical and hence it is itself a finite model for $K(\Gamma, 1)$. \end{proof}

Now we collect all of these auxiliary results to present:

\begin{Theorem}
Let $M$ be a higher graph manifold of dimension $\geq 6$. Assume that all of the pieces of $M$ are trivial bundles. Then $M$ satisfies the Borel conjecture.
\end{Theorem}

\begin{proof} 
Notice that these higher graph manifolds satisfy all the hypothesis of Lemma \ref{modifiedstrategy}. As has been shown in Lemmas \ref{asdimsFDC}, \ref{FIC-for-pieces-and-walls}, \ref{whitehead}, \ref{squareroot}, and \ref{finitemodel}.\end{proof}

\end{document}